\documentclass[twoside, 12pt]{amsart}
\usepackage[letterpaper,hmargin=1.25in,vmargin=1.3in]{geometry}
\usepackage{amscd}
\usepackage{verbatim}
\usepackage{color}
\usepackage{comment}
\usepackage{array}
\usepackage{multirow}

\input xy
\xyoption{all}

\usepackage{amssymb}
\usepackage{hyperref}

\DeclareMathAlphabet{\cat}{OT1}{cmss}{m}{sl}

\newtheorem*{theorem*}{Theorem}
\newtheorem{theorem}{Theorem}[section]

\newtheorem{proposition}[theorem]{Proposition}
\newtheorem{lemma}[theorem]{Lemma}
\newtheorem{corollary}[theorem]{Corollary}

\theoremstyle{definition}
\newtheorem{remark}[theorem]{Remark}

\newcommand{\tens}{\otimes}

\newcommand{\gmu}{\boldsymbol{\mu}}

\newcommand{\diag}{\mathrm{diag}}

\newcommand{\Ker}{\operatorname{Ker}}

\newcommand{\ind}{\operatorname{\hspace{0.3mm}ind}}

\newcommand{\Br}{\operatorname{Br}}
\newcommand{\Spec}{\operatorname{Spec}}

\newcommand{\GL}{\operatorname{GL}}
\newcommand{\gPGL}{\operatorname{\mathbf{PGL}}}

\newcommand{\gGSp}{\operatorname{\mathbf{GSp}}}
\newcommand{\gSL}{\operatorname{\mathbf{SL}}}

\newcommand{\gSO}{\operatorname{\mathbf{SO}}}
\newcommand{\gGamma}{\operatorname{\mathbf{\Gamma}}}
\newcommand{\gGL}{\operatorname{\mathbf{GL}}}
\newcommand{\gm}{\operatorname{\mathbb{G}}_m}

\newcommand{\gSpin}{\operatorname{\mathbf{Spin}}}

\newcommand{\ed}{\operatorname{ed}}

\newcommand{\rank}{\operatorname{rank}}

\newcommand{\red}{\operatorname{red}}

\renewcommand{\P}{\mathbb{P}}
\newcommand{\Z}{\mathbb{Z}}

\title[Essential dimension of semisimple groups of type $B$] 
{Essential dimension of semisimple groups of type $B$}

\author
[S.~Baek, Y.~Kim] {Sanghoon Baek, Yeongjong Kim}

\address[Sanghoon Baek]{Department of Mathematical Sciences, 
	KAIST,
	291 Daehak-ro, Yuseong-gu,
	Daejeon 34141,
	Republic of Korea}
\email{sanghoonbaek@kaist.ac.kr}
\urladdr{http://mathsci.kaist.ac.kr/~sbaek/}

\address[Yeongjong Kim]{Department of Mathematical Sciences, 
	KAIST,
	291 Daehak-ro, Yuseong-gu,
	Daejeon 34141,
	Republic of Korea}
\email{kimyj@kaist.ac.kr}

\begin{document}
	
\begin{abstract}
We determine the essential dimension of an arbitrary semisimple group of type $B$ of the form
\[G=\big(\gSpin(2n_{1}+1)\times\cdots \times \gSpin(2n_{m}+1)\big)/\gmu\]
over a field of characteristic $0$, for all $n_{1},\ldots, n_{m}\geq 7$, and a central subgroup $\gmu$ of $\gSpin(2n_{1}+1)\times\cdots \times \gSpin(2n_{m}+1)$ not containing the center of $\gSpin(2n_i+1)$ as a direct factor. We also find the essential dimension of $G$ for each of the following cases, where either $n_{i}=1$ for all $i$ or $m=2$, $n_{1}=1$, $2\leq n_{2}\leq 3$, $\gmu$ is the diagonal central subgroup for both cases.
\end{abstract}

\maketitle

\section{Introduction}

The essential dimension, introduced by Buhler and Reichstein in \cite{BR}, is a numerical invariant which measures the complexity of torsors under an algebraic group over a field. More precisely, given an algebraic group $G$ over a field $F$, consider the Galois cohomology functor
\[H^{1}(-,G):\cat{Fields}_{F}\to \cat{Sets},\]
which sends a field extension $K$ over $F$ from the category $\cat{Fields}_{F}$ of field extensions of $F$ to the set $H^{1}(K,G)$ of isomorphism classes of $G$-torsors over $\Spec(K)$ in the category of sets. The \emph{essential dimension} $\ed(\eta)$ of a $G$-torsor $\eta\in H^{1}(K,G)$ is defined by the minimal transcendental degree of an intermediate field $F\subset L\subset K$ such that $\eta$ lies in the image of $H^{1}(L,G)\to H^{1}(K,G)$. The essential dimension of $G$, denoted $\ed(G)$, is defined by the maximal value of $\ed(\eta)$, where the maximum ranges over all field extensions $K/F$ and all $\eta\in H^{1}(K,G)$. Indeed, the notion of essential dimension can be defined for any algebraic structure given by a functor from $\cat{Fields}_{F}$ to $\cat{Sets}$ and naturally extended to the categories fibered in groupoids \cite{BerhuyFavi, Merkurjev}.

In general, the accurate computation of the essential dimension of $G$ is a challenging problem. Even when $G$ is a semisimple group of a certain Dynkin type, the complexities of the corresponding torsors differ considerably depending on the central isogenies of $G$. For instance, by Hilbert theorem $90$ the essential dimension of $\gSL(n)$ is trivial for any $n\geq 2$, but on the other hand the essential dimension of the corresponding adjoint group $\gPGL(n)$ is at least $2$ for any $n$ and indeed the computation of $\ed(\gPGL(n))$ remains a major open problem. Similarly, it is known that the essential dimension of the special orthogonal group $\gSO(n)$ is equal to $n-1$ over a field of characteristic different from $2$ for any $n\geq 3$ \cite{Reichstein00}, but on the other hand the essential dimension of its simply connected cover $\gSpin(n)$ grows exponentially with $n$ \cite{BRV}.

The spin group $\gSpin(n)$, i.e., a split simply connected group of type $B$ or $D$, has been of particular interest to the research on the essential dimension and indeed, except for obvious cases, such groups are almost unique simple groups whose essential dimensions are completely calculated. The computation of the essential dimension of $\gSpin(n)$ is divided into two parts, $n\geq 15$ and $7\leq n\leq 14$ depending on the generic freeness of the corresponding (half-)spin representation of $\gSpin(n)$ (see Section \ref{sec:upperB}).

Assume $n\geq 15$. In \cite{BRV}, the essential dimension $\ed(\gSpin(n))$ over a field of characteristic $0$ was first computed by Brosnan, Reichstein, and Vistoli in the case where $n$ is not divisible $4$. The remaining cases ($n$ is divisible by $4$) were completed by Chernousov and Merkurev \cite{CM}. Recently, Garibaldi, Guralnick and Totaro showed that these results also hold over an arbitrary field in \cite{GG} (the case of characteristic not $2$) and in \cite{Totaro} (the case of characteristic $2$). For $7\leq n\leq 14$, the essential dimension $\ed(\gSpin(n))$ was computed based on a case-by-case analysis by Rost and Garibaldi over a field of characteristic not $2$ \cite{Garibaldi}. Totaro also showed that the same result holds over a field of characteristic $2$ in the case where $7\leq n\leq 10$ \cite{Totaro}.

In view of the successful computation of the essential dimension of spin groups, it is natural to ask whether this result can be extended to semisimple groups (not necessarily simple), i.e., we propose to consider an arbitrary semisimple group
\[G=\big(\gSpin(l_{1})\times\cdots \times \gSpin(l_{m})\big)/\gmu\]
of types $B$ and $D$, where $\gmu$ is a central subgroup of $\gSpin(l_{1})\times\cdots \times \gSpin(l_{m})$. In the present paper we prove that the essential dimension of a semisimple group $G$ of type $B$ (with $l_{i}\geq 15$ odd) is determined by the minimal sum of the tensor products of spin representations of $\gSpin(l_{i})$, where the tensor product is taken over a basis of the subgroup given by $\gmu$, and consequently we compute its essential dimension (see Theorem \ref{thm:main}). We also compute the essential dimension of each of the following three groups
\begin{equation}\label{simplestgroups}
G=\gSpin(3)^{m}/\gmu,\quad (\gSpin(3)\times \gSpin(5))/\gmu,\quad (\gSpin(3)\times \gSpin(7))/\gmu,
\end{equation}
where $\gmu$ is the diagonal central subgroup in each case (see Propositions \ref{prop:sec51} and \ref{prop:smallesttwo}), which can be viewed as the simplest cases of semisimple groups of type $B$ whose essential dimensions are not covered by our main theorem. By adopting a similar approach, together with Lemma \ref{lem:smalllower1} and Lemma \ref{lem:smalllower2}, we expect that the computation of the essential dimension of an arbitrary semisimple group of type $B$ can be completed. In particular, our methods developed in this paper are directly applicable to the case of a semisimple group of type $D$ (see Remark \ref{rmk:typeD}), which will be treated in a sequel paper. We also remark that the essential dimension of the quotient of a product of general linear groups by a central subgroup was studied in \cite{CR}.

This paper is structured as follows. In the following section, we first begin with a simple result on the essential dimension of a direct product of simply connected groups of types $B$ and $D$ and then we present our main results providing the essential dimension of a general semisimple group of type $B$. In Sections \ref{sec:upperB} and \ref{sec:lower}, we prove our main results, establishing upper and lower bounds on the proofs. In the last section we compute 
the essential dimension of the groups in (\ref{simplestgroups}). We also provide lower bounds on the essential dimension of quotients of small groups by the maximal central subgroups in Lemma \ref{lem:smalllower1} and Lemma \ref{lem:smalllower2}.

In the sequel we shall use the following notation. Unless explicitly stated otherwise, the base field $F$ is assumed to have characteristic zero (indeed, the characteristic $0$ assumption is necessary for Propositions \ref{prop:generic} and \ref{prop:mainupperB} in Section \ref{sec:upperB}). We write $\gmu_{n}$ and $\gm$ for the group of $n$th roots of unity and the multiplicative group over $F$, respectively. For an algebraic group $G$ we denote by $Z(G)$ and $G^{*}$ the center of $G$ and the character group of $G$, respectively. We shall write $\langle a_{1},\ldots, a_{n}\rangle$ for the diagonal quadratic form $a_{1}x_{1}^{2}+\cdots +a_{n}x_{n}^{2}$. We denote by $(a, b)$ the quaternion $F$-algebra generated by $\{1, i, j, k\}$ with the relations $i^{2}=a, j^{2}=b, ij=k=-ji$ for some $a,b\in F^{\times}$. We denote by $[n]$ the set $\{1,2,\ldots,n\}$. Finally, we write $\Br(F)$ for the Brauer group of $F$. Given a central simple $F$-algebra $A$, we write $\ind(A)$ for the degree of the unique division algebra of the class of $A$ in $\Br(F)$.

\medskip

\paragraph{\bf Acknowledgements.} 
This work was supported by Samsung Science and Technology Foundation under Project Number SSTF-BA1901-02.

\section{Main results}
Every semisimple group of type $B$ can be written as 
\begin{equation}\label{semisimpleB}
\big(\prod_{i\in I}\gSpin(2n_i+1)\big)/\gmu\times \prod_{j\in J}\gSO(2n_j+1)
\end{equation}
for some index sets $I$ and $J$, where $\gmu$ is a central subgroup not containing the center of $\gSpin(2n_i+1)$ as a direct factor for any $i\in I$. We call a semisimple group of type $B$ \emph{reduced} if the set $J$ in $(\ref{semisimpleB})$ is empty (i.e., it does not contain $\gSO(2n_j+1)$ factors). In this paper, we shall focus on the essential dimension of reduced semisimple groups of type $B$.

The essential dimension of a direct product of spin groups is determined by the sum of its direct components. The proof is a simple extension of the proof of \cite[Theorem 3-3]{BRV}.

\begin{proposition}\label{prop:direct}
Let $G_{i}=\gSpin(n_{i})$ over $F$, where $n_{i}\geq 15$ is not divisible by $4$ for all $1\leq i\leq m$. Then, we have
\[\ed(\prod_{i=1}^{m}G_{i})=\sum_{i=1}^{m}\ed(G_{i}).\]
\end{proposition}
\begin{proof}
By \cite[Lemma 1.11]{BerhuyFavi}, we have $\ed(\prod_{i=1}^{m}G_{i})\leq \sum_{i=1}^{m}\ed(G_{i})$, thus it suffices to prove the opposite inequality.

It follows by \cite[(3-5)]{BRV} that $\ed(G_{i})=\ed(\Delta(i))-\dim(G_{i})$, where $\Delta(i)$ denotes the extraspecial $2$-subgroup of $G_{i}$ (see Section \ref{sec:small}). Hence, by \cite[Lemma 2-2]{BRV}, we have $$\ed(\prod_{i=1}^{m}G_{i})\geq \ed(\prod_{i=1}^{m}\Delta(i))-\sum_{i=1}^{m}\dim(G_{i}).$$
By \cite[Theorem 5.1]{KM}, we get $\ed(\prod_{i=1}^{m}\Delta(i))=\sum_{i=1}^{m}\ed(\Delta(i))$, thus the inequality $\ed(\prod_{i=1}^{m}G_{i})\geq \sum_{i=1}^{m}\ed(G_{i})$ holds. \end{proof}
\begin{remark}
Proposition \ref{prop:direct} also holds if $3\leq n_{i}\leq 14$ for all $1\leq i\leq m$, which can be proved using Theorem \ref{thm:lowersmall}.
\end{remark}

Next, we show that the essential dimension of a semisimple group of type $B$ given as the quotient by a maximal central subgroup is determined by the tensor product of spin representations as follows, which extends the preceding result on a direct product of spin groups.

\begin{proposition}\label{prop:mainprop}
Let $\tilde{G}=\prod_{i=1}^{m}\gSpin(2n_{i}+1)$ be a non-small semisimple group of type $B$ over $F$ with $n_{i}\geq 1$, i.e., $\tilde{G}$ is not equal to one of the groups in $(\ref{small1})$, $(\ref{small2})$, $(\ref{small3})$, and  $(\ref{small4})$. Let $G=\tilde{G}/\gmu$, where $\gmu$ is a maximal subgroup of $Z(\tilde{G})=(\gmu_{2})^{m}$ given by the kernel of the product map $(\gmu_{2})^{m}\to \gmu_{2}$. Then, 
\[\ed(G)=\dim(V)-\dim(G)=2^{(n_{1}+\cdots +n_{m})}-\sum_{i=1}^{m}2n_{i}^{2}+n_{i},\]
where $V$ is the tensor product of spin representations of $\tilde{G}$.
\end{proposition}
\begin{proof}
It follows by Corollary \ref{cor:maximal} and Proposition \ref{prop:lower}.
\end{proof}

Now we determine the essential dimension of an arbitrary reduced semisimple group of type $B$, which generalize Propositions \ref{prop:direct} and \ref{prop:mainprop} to an arbitrary central subgroup.

\begin{theorem}\label{thm:main}
Let $G=(\prod_{i=1}^{m}G_{i})/\gmu$ be a reduced semisimple group of type $B$ over $F$, where $G_{i}=\gSpin(2n_{i}+1)$ and $\gmu$ is a central subgroup such that for all $i$ either
$n_{i}\geq 7$ or $n_{i}\geq 3$ and $G_{i}$ is not a direct factor of $G$. Let $R$ be the subgroup of $Z(\prod_{i=1}^{m}G_{i})^{*}=(\Z/2\Z)^{m}$ whose quotient is the character group $\gmu^{*}$. For any $r=(r_{1},\ldots, r_{m})\in R$, we set $V[r]=\bigotimes_{i \in \operatorname{supp}(r)} V(i)$, where $V(i)$ denotes the spin representation of $G_{i}$ and  $\operatorname{supp}(r)=\{i\in[m]\,|\,r_i\neq 0\}$. Then,
\begin{align*}\ed(G)&=\min_{B}\big\{\sum_{r\in B} \dim V[r]\big\}-\dim(G)\\
	&=\min_{B}\Big\{\sum_{r\in B} \big(\prod_{i \in\operatorname{supp}(r)}2^{n_{i}}\big)\Big\}-\sum_{i=1}^{m}2n_{i}^{2}+n_{i},
\end{align*}
where the minimum ranges over all bases $B$ of $R$.
\end{theorem}
\begin{proof}
It follows by Corollary \ref{cor:upperB} and Proposition \ref{prop:lower}.
\end{proof}

We remark that our assumption on the integers $n_{i}$ in Theorem \ref{thm:main} can be slightly generalized (See Remark \ref{rmk:assumption}). 

A basis $B$ of $R$ in Theorem \ref{thm:main} is called \emph{minimal} if it gives the essential dimension of $G$. As the number of bases of $R$ in Theorem \ref{thm:main} is finite, we can find a minimal basis of $R$. However, in practice, one can use an efficient method called a greedy algorithm \cite[Lemma 1.8.3]{Oxley} as follows: For a nonzero element $r\in R$, the corresponding value $2^{\sum_{i\in \operatorname{supp}(r)}n_i}$ is called the \emph{weight} of $r$. We enumerate all nonzero vectors $r_1,\ldots,r_N$ with $N=2^{m-k}-1$ so that the corresponding weights are in non-decreasing order. Set $B=\emptyset$ initially and then add $r_i$ to $B$ successively from $i=1$ whenever $B\cup \{r_i\}$ is linearly independent. Then, we obtain a minimal basis $B$ of $R$. For instance, we have

\begin{corollary}
Let $G=(\prod_{i=1}^{m}\gSpin(2n_{i}+1))/\gmu$, where $m\geq 2$, $n_i\geq 3$ for all $i$, and $\gmu$ is the diagonal central subgroup. Then,
\begin{equation*}
    \ed(G) = \sum_{i\neq j}^m 2^{n_j+n_i} - \sum_{i=1}^{m}2n_{i}^{2}+n_{i},
\end{equation*}
where $n_j$ is a minimum of all integers $n_i$.
\end{corollary}
\begin{proof}
We may assume that $n_1\leq\cdots\leq n_m$. Applying the argument as above, we obtain a minimal basis $\{(1,1,0,\ldots,0),(1,0,1,0,\ldots,0),\ldots,(1,0,\ldots,0,1)\}$ of $R$. 
\end{proof}

\section{Upper bounds for semisimple groups of type $B$}\label{sec:upperB}

Let $G$ be an algebraic group over $F$ acting on a finite dimensional vector space $V$. We say that a subgroup $H\subset G$ is a \emph{generic stabilizer} (or a \emph{stabilizer in general position}) for the action if there is a dense open subset $U$ of $V$ such that the stabilizer of any $v\in U$ is conjugate to $H$. In particular, we say that the $G$-action on $V$ is \emph{generically free} if the trivial subgroup of $G$ is a generic stabilizer for this action. Every generically free $G$-action yields an upper bound of the essential dimension of $G$ as follows.

\begin{lemma}\cite[Theorem 3.4]{Reichstein00},\cite[Lemma 4.11]{BerhuyFavi}\label{lem:generically}
Let $G$ be an algebraic group over $F$. If $G$ acts generically free on a vector space $V$ over $F$, then $\ed(G)\leq \dim(V)-\dim(G)$.
\end{lemma}

Consider a split simply connected semisimple group $\tilde{G}$ of type $B$ over $F$ of one of the following: 
\begin{align}
&\gSpin(2n+1), 1\leq n\leq 6, \label{small1}\\
& \gSpin(3)\times \gSpin(2n+1), 1\leq n\leq 5,\, \gSpin(5)\times \gSpin(5),\,\gSpin(5)\times \gSpin(7),\label{small2}\\
&\gSpin(3)^{3},\, \gSpin(3)^{2}\times \gSpin(5),\, \gSpin(3)^{2}\times \gSpin(7),\label{small3}\\
& \gSpin(3)^{4}.\label{small4} 
\end{align}

Such groups will be called \emph{small} (otherwise called \emph{non-small}). The generic freeness for quotients of non-small semisimple groups of type $B$ by the maximal central subgroups is proven as follows.

\begin{proposition}\cite{BRV, Ela2, Popov2}\label{maximalgeneric}\label{prop:generic}
Let $\tilde{G}=\prod_{i=1}^{m}\gSpin(2n_{i}+1)$ be a non-small semisimple group of type $B$ over $F$ with $n_{i}\geq 1$. Let $G=\tilde{G}/\gmu$, where $\gmu$ is a maximal subgroup of $Z(\tilde{G})=(\gmu_{2})^{m}$ given by the kernel of the product map $(\gmu_{2})^{m}\to \gmu_{2}$. Then, the action of $G$ on the tensor product $\bigotimes_{i=1}^{m}V(i)$ of the spin representations $V(i)$ of $\gSpin(2n_{i}+1)$ is generically free.
\end{proposition}
\begin{proof}
If $\tilde{G}=G$ is simple, the statement follows by \cite[Lemma 3-7]{BRV}. Now we assume $m\geq 2$. Let $V=\bigotimes_{i=1}^{m}V(i)$ and let $\rho: \tilde{G}\to \GL(V)$ be the tensor product representation of $\tilde{G}$. By \cite[Theorems $5$, $6$, Tables 5, 6]{Ela2} and \cite[Theorem 1, Tables $0$, $1$]{Popov2}, for every non-small $\tilde{G}$ the stabilizer subgroup $S$ of $v\in V$ in general position is contained in $Z(\tilde{G})$. Hence,  $S=\Ker(\rho)\subset Z(\tilde{G})$. Therefore, by Schur's lemma we have $\Ker(\rho)=\gmu$ and the statement follows.\end{proof}

\begin{remark}
Indeed, for non-simple small groups $\tilde{G}$, the identity component $S^{\circ}$ of a generic stabilizer $S$ in $G=\tilde{G}/\gmu$, where $\gmu$ is the maximal central subgroup as in Proposition \ref{prop:generic}, is given as follows \cite{Ela2, Popov2}: 

\smallskip

\begin{center}
	
	\begin{tabular}{ c | c || c | c }
		\hline
		$\tilde{G}$ & $S^{\circ}$ & $\tilde{G}$ & $S^{\circ}$ \rule{0pt}{2.5ex} \\
		\hline
		
		$\gSpin(3)\times \gSpin(5)$ & $\gSL(2)^{2}$ & $\gSpin(5)\times \gSpin(7)$ & \,$\gm^{2}$ \rule{0pt}{2.5ex}\\
		
		$\gSpin(3)\times \gSpin(7)$  & $\gm\times \gSL(3)$ & $\gSpin(3)^{3}$  & \,$\gm^{2}$ \rule{0pt}{2.5ex}\\
		
		$\gSpin(3)\times \gSpin(9)$& $\gm\times \gSL(3)$ &  $\gSpin(3)^{2}\times \gSpin(5)$ & \,$\gm^{2}$ \rule{0pt}{2.5ex}\\
		
		$\gSpin(3)\times \gSpin(11)$ & $1$  & $\gSpin(3)^{2}\times \gSpin(7)$ & \,$1$ \rule{0pt}{2.5ex}\\
		
		$\gSpin(5)\times \gSpin(5)$ & $\gSL(2)^{2}$ & $\gSpin(3)^{4}$ & \,$1$ \rule{0pt}{2.5ex}\\
		
		\hline
	\end{tabular}

\end{center}

\smallskip

In cases where $S^{\circ}$ is trivial in the table above, the generic stabilizer $S$ is isomorphic to $\gmu_{2}^{2}$. Note that for $\tilde{G}=\gSpin(3)\times\gSpin(3)$ we have $S=\gSL(2)$.
\end{remark}

It follows by Lemma \ref{lem:generically} that
\begin{corollary}\label{cor:maximal}
Let $G=\tilde{G}/\gmu$, where $\tilde{G}$ and $\gmu$ are the groups defined as in Proposition \ref{maximalgeneric}. Then, $\ed(G)\leq \dim(V)-\dim (G)$, where $V=\bigotimes_{i=1}^{m}V(i)$ is the tensor product of spin representations of $\tilde{G}$.
\end{corollary}

We shall need the following lemma.

\begin{lemma}\cite[Proposition 8]{VLPopov}\label{lem:transitive}
Let an algebraic group $G$ over $F$ act on finite dimensional vector spaces $V$ and $W$. Let $K\subset H\subset G$ be subgroups such that $H$ is a generic stabilizer for the $G$-action on $V$ and $K$ is a generic stabilizer for the $H$-action on $W$. Then, $K$ is a generic stabilizer for the $G$-action on $V\oplus W$. 
\end{lemma}

We generalize the generically free action of a semisimple group of type $B$ given as the quotient by the maximal central subgroup in Poposition \ref{prop:generic} to the quotient by an arbitrary central subgroup.

\begin{proposition}\label{prop:mainupperB}
Let $G=(\prod_{i=1}^{m}G_{i})/\gmu$ be a reduced semisimple group of type $B$ over $F$, where $G_{i}=\gSpin(2n_{i}+1)$ and $\gmu$ is a central subgroup such that for all $i$ either
$n_{i}\geq 7$ or $n_{i}\geq 3$ and $G_{i}$ is not a direct factor of $G$. Let $R$ be the subgroup of $Z(\prod_{i=1}^{m}G_{i})^{*}=(\Z/2\Z)^{m}$ whose quotient is the character group $\gmu^{*}$. For any $r=(r_{1},\ldots, r_{m})\in R$, we set $V[r]=\bigotimes_{i\in\operatorname{supp}(r)} V(i)$, where $\operatorname{supp}(r)=\{i\in [m]\,|\,r_i\neq 0\}$ and $V(i)$ denotes the spin representation of $G_{i}$. Then, $G$ acts generically freely on $\bigoplus_{r\in B}V[r]$ for any basis $B$ of $R$.
\end{proposition}
\begin{proof}
Let $\gmu\simeq (\gmu_{2})^{k}$ for some $0\leq k\leq m$ and let $R$ be the subgroup of $Z(\prod_{i=1}^{m}G_{i})^{*}:=\bigoplus_{i=1}^{m}(\Z/2\Z)e_{i}$ such that $\gmu^{*}=\big(\bigoplus_{i=1}^{m}(\Z/2\Z)e_{i}\big)/R$. Then, $\dim_{\Z/2\Z}(R)=m-k$.

Let $\{r[1],\ldots,r[m-k]\}$ be a basis of $R$. Then, we have
\[\gmu=\{(a_{1},\ldots, a_{m})\in (\gmu_{2})^{m} \,|\, \prod_{i\in \operatorname{supp}(r[p])}a_i=1,\, 1\leq p\leq m-k\},\]
where $\operatorname{supp}(r[p])=\{i\in [m]\,|\,(r[p])_{i}\neq 0\}$.

For each $1\leq p\leq m-k$, we set
\begin{align*}
G[p]&= \prod_{i\in \operatorname{supp}(r[p])}G_i,\\
V[p]&= \bigotimes_{i\in \operatorname{supp}(r[p])}V(i), \text{ and }\\
\gmu[p]&=\{(a_{1},\ldots, a_{m})\in (\gmu_{2})^{m}\,|\, \prod_{i\in \operatorname{supp}(r[p])}a_i=1\},
\end{align*}
where $G_{i}=\gSpin(2n_{i}+1)$ and $V(i)$ denotes the spin representation of $G_{i}$.

Let $\tilde{G}=\prod_{i=1}^{m}G_{i}$. We view each $G[p]$ as a subgroup of $\tilde{G}$ and each $\gmu[p]$ as a subgroup of $Z(\tilde{G})$. If $\nu$ is a subgroup of $Z(\tilde{G})$, we define $\tilde{G}\star\nu$ to be a subgroup of $\tilde{G}$ given by replacing $Z(G_i)$ by $G_i$ for all $i$ such that $Z(G_i)$ is a direct factor of $\nu$. Let
\[H[p]=\frac{\tilde{G}\star\big(\cap_{i=1}^{p} \gmu[i]\big)}{\gmu}\]
for $1\leq p\leq m-k$. Then, 
\[\{1\}=H[m-k]\subset H[m-k-1]\subset \cdots \subset H[2]\subset H[1]\subset H[0]:=G.\]

Consider the $H[p-1]$-action on $V[p]$:
\[H[p-1]=\frac{\tilde{G}\star\big(\cap_{i=1}^{p-1} \gmu[i]\big)}{\gmu}\to \frac{G[p]}{G[p]\cap \gmu[p]}\stackrel{\phi_{p}}{\longrightarrow} \GL(V[p]), \]
where the first map is the projection and the second map $\phi_{p}$ is induced by the spin representation of $G[p]$. Then, since by Proposition \ref{prop:generic} and our assumption on $G$, the quotient group $G[p]/(G[p]\cap \gmu[p])$ acts generically freely on $V[p]$, we see that $H[p]$ is a generic stabilizer of the $H[p-1]$-action on $V[p]$. Therefore, by Lemma \ref{lem:transitive} we conclude that $G$ acts generically freely on $\bigoplus_{p=1}^{m-k}V[p]$.
\end{proof}

By Lemma \ref{lem:generically}, we obtain the following upper bounds for semisimple groups of type $B$.

\begin{corollary}\label{cor:upperB}
	Let $G=(\prod_{i=1}^{m}G_{i})/\gmu$ and let $V[r]=\bigotimes_{i\in\operatorname{supp}(r)} V(i)$ for any $r\in R$, where $G_{i}$, $\gmu$, $R$ are the groups defined as in Proposition \ref{prop:mainupperB} and $V(i)$ denotes the spin representaion of $G_{i}$. Then, for any basis $B$ of $R$ we have $$\ed(G)\leq \big(\sum_{r\in B} \dim V[r]\big)-\dim G.$$	
\end{corollary}

\begin{remark}\label{rmk:assumption}
(1) In the proof of Proposition \ref{prop:mainupperB}, we make use of our assuption that $n_{i}\geq 7$ or $n_{i}\geq 3$ and $G_{i}$ is not a direct factor of $G$ to guarantee that the quotient group $G[p]/(G[p]\cap \gmu[p])$ acts generically freely on $V[p]$. Hence, the assumption can be slightly relaxed by the following assumption that  each $G[p]$ is not small for a basis $B$ of $R$. Moreover, if $B$ is a minimal basis, then the upper bound obtained from Corollary \ref{cor:upperB} coincides with the lower bound obtained from Proposition \ref{prop:lower}. For example, consider
\begin{equation*}
    G = \gSpin(3)\times\gSpin(5)\times\gSpin(7)\times\gSpin(15)/\langle (-1,-1,1,1),(-1,1,-1,1)\rangle.
\end{equation*}
Then, for a basis $B=\{(1,1,1,0),\,(0,0,0,1)\}$ the groups $\gSpin(3)\times\gSpin(5)\times\gSpin(7)$ and $\gSpin(15)$ are not small, thus 
$\ed(G)\leq 53$ by Corollary \ref{cor:upperB}. Indeed, this upper bound coincides with the lower bound of Proposition \ref{prop:lower} so $\ed(G)=53$.\end{remark}

\begin{remark}\label{rmk:typeD}
The underlying ideas of Proposition \ref{prop:mainupperB} is applicable to type $D$ as follows. Let $G=(\prod_{i=1}^{m}G_{i})/\gmu$, where $G_{i}=\gSpin(2n_{i})$ with $n_i>3$ odd and $\gmu$ is a central subgroup not containing $Z(G_i)=\gmu_{4}$ as a direct factor. Let $R$ be the subgroup of $Z(\prod_{i=1}^{m}G_{i})^{*}$ whose quotient is $\gmu^{*}$ and let $B=\{r[1],\ldots,r[k]\}$ be its minimal generating set. For each $1\leq p\leq k$, let $G[p]$ be the same as in the proof of Proposition \ref{prop:mainupperB}, 
\begin{align*}
V[p]&= 
\Big(\bigotimes_{i\in \operatorname{supp}_1(r[p])}V_{+}(i)\Big)\otimes
\Big(\bigotimes_{i\in \operatorname{supp}_2(r[p])}W(i)\Big)\otimes
\Big(\bigotimes_{i\in \operatorname{supp}_3(r[p])}V_{-}(i)\Big), \text{ and }\\
\gmu[p]&=\{(a_{1},\ldots, a_{m})\in (\gmu_{4})^{m}\,|\, \prod_{l=1}^3\prod_{i\in \operatorname{supp}_l(r[p])}a_i^l=1\},
\end{align*}
where $\operatorname{supp}_l(r[p])=\{i\in[m]\,|\, (r[p])_{i}=l\}$ for $1\leq l\leq 3$ and $V_{\pm}(i), W(i)$ are two half-spin representations and the vector representation of $G_i$ respectively. Then, the proof of Proposition \ref{prop:mainupperB} shows that if $G[p]/G[p]\cap\gmu[p]$ acts generically freely on $V[p]$ for all $p$, then $G$ acts generically freely on $\bigoplus_{p=1}^{k} V[p]$. Using \cite{Ela1, Ela2, Popov1, Popov2} one can check if the faithful action of $G[p]/(G[p]\cap\gmu[p])$ on $V[p]$ is generically free. For instance, let $G$ be as above with $R=\langle (1,1,0,\ldots,0),(1,0,1,0,\ldots,0),\ldots,(1,0,\ldots,0,1)\rangle$. Then, by \cite{Ela2, Popov2}
$G$ acts generically freely on $\bigoplus_{i=2}^m V_{+}(1)\otimes V_{+}(i)$, thus by Lemma \ref{lem:generically}, we have $\ed(G)\leq \sum_{j=2}^m 2^{n_1+n_j-2} -\sum_{i=1}^m 2n_i^2-n_i$. 
\end{remark}

\section{Lower bounds for semisimple groups of type $B$}\label{sec:lower}

Consider an exact sequence of algebraic groups over $F$
\[1\to \gmu\to G\to H\to 1,\]
where $\gmu$ is a central subgroup of $G$ such that $\gmu\simeq (\gmu_{p})^{k}$ for some prime $p$ and $k\geq 1$. Then, for any character $\chi:\gmu\to \gm$, we have the induced sequence
\begin{equation}\label{lowerexactseq}
H^{1}(F,H)\stackrel{\partial}{\to} H^{2}(F,\gmu)\stackrel{\chi_{*}}{\to} H^{2}(F,\gm)=\Br(F),
\end{equation}
where $\Br(F)$ denotes the Brauer group of $F$. For $\eta\in H^{1}(F,H)$, consider the homomorphism
\[\phi_{\eta}:\gmu^{*}\to \Br(F),\quad \chi\mapsto \chi_{*}(\partial(\eta))    \]
induced by (\ref{lowerexactseq}). Then, we obtain

\begin{theorem}\cite[Theorem 4.1]{Reichstein}, \cite[Theorems 6.1, 6.2]{Merkurjev}\label{thm:lowerbounds}
Let $G$ be an algebraic group over $F$ and let $\gmu$ be its central subgroup such that $\gmu\simeq (\gmu_{p})^{k}$ for some prime integer $p$ and some positive integer $k$. Then, for any $\eta\in H^{1}(F,G/\gmu)$ we have
\[\ed(G)\geq \min\{\sum_{\chi\in B}\ind\phi_{\eta}(\chi)\}-\dim(G),\]
where the minimum ranges over all bases $B$ of $\gmu^{*}$.
\end{theorem}

To apply theorem to semisimple groups of type $B$, we shall need the following lemma.
\begin{lemma}\label{lem:evenclifford}
For any $n\geq 1$, there exist a field extension $E/F$ and a $(2n+1)$-dimensional quadratic form $q$ over $E$ such that $\ind C_{0}(q)=2^{n}$, where $C_{0}(q)$ denotes the even Clifford algebra of $q$.
\end{lemma}
\begin{proof}
We may assume that $F$ is algebraically closed. Let $q=\langle 1, x_{1},\ldots, x_{2n}\rangle$ over $E=F(x_{1},\ldots, x_{2n})$, where $x_{1},\ldots, x_{2n}$ are algebraically independent variables over $F$. Then, we have
\[C_{0}(q)\simeq (x_{1}, x_{2})\tens \cdots \tens (x_{2n-1},x_{2n})\]
where each $(x_{i},x_{i+1})$ denotes the quaternion algebra. As $C_{0}(q)$ is division $E$-algebra, we have $\ind C_{0}(q)=2^{n}$.
\end{proof}

\begin{corollary}\label{cor:evenclifford}
Let $G=(\prod_{i=1}^{m}\gSpin(2n_{i}+1))/\gmu$ over $F$, where $\gmu$ is a central subgroup, and let $R$ be the subgroup of $Z(\prod_{i=1}^{m}\gSpin(2n_{i}+1))^{*}=(\Z/2\Z)^{m}$ whose quotient is the character group $\gmu^{*}$. Let $H=\prod_{i=1}^{m}\gSO(2n_{i}+1)$. Then, there exists an $H$-torsor $(q_{1},\ldots, q_{m})$ over a field extension $E/F$ such that the $E$-algebra $\bigotimes_{i=1}^{m}r_{i}C_{0}(q_{i})$ is division for every nonzero $r=(r_{1},\ldots, r_{m})\in R$.
\end{corollary}
\begin{proof}
It suffices to show that there exists an $H$-torsor $(q_{1},\ldots, q_{m})$ over a field extension $E/F$ such that each $E$-algebra $C_{0}(q_{i})$ is division for every $1\leq i\leq m$ and the $E$-algebra $\bigotimes_{i=1}^{m}C_{0}(q_{i})$ is division, which follows by Lemma \ref{lem:evenclifford}.
\end{proof}

Now we obtain the lower bounds for the essential dimension of semisimple groups of type $B$.
\begin{proposition}\label{prop:lower}
Let $G_{i}=\gSpin(2n_{i}+1)$ over $F$ and let $G=(\prod_{i=1}^{m}G_{i})/\gmu$, where $\gmu$ is a central subgroup. Let $R$ be the subgroup of $Z(\prod_{i=1}^{m}G_{i})^{*}=(\Z/2\Z)^{m}$ whose quotient is the character group $\gmu^{*}$. Then,
\[\ed(G)\geq \min_{B}\Big\{\sum_{r\in B} \big(\prod_{i\in\operatorname{supp}(r)}2^{n_{i}}\big)\Big\}-\dim(G),\]
where the minimum ranges over all bases $B$ of $R$.
\end{proposition}
\begin{proof}
Consider the exact sequence
\[ 1\to \frac{\gmu_{2}^{m}}{\gmu}\to G\to H\to 1,   \]
where $H=\prod_{i=1}^{m}\gSO(2n_{i}+1)$. Let $\eta=(q_{1},\ldots, q_{m})\in H^{1}(E,H)$ as in Corollary \ref{cor:evenclifford}. Then, we have $R=\big((\gmu_{2})^{m}/\gmu\big)^{*}$ and the homomorphism
\[\phi_{\eta}: R\to \Br(E),\quad r=(r_{1},\ldots, r_{m})\mapsto \bigotimes_{i=1}^{m}r_{i}C_{0}(q_{i}).\]
By Corollary \ref{cor:evenclifford}, we obtain 
\[\ind\big(\bigotimes_{i=1}^{m}r_{i}C_{0}(q_{i})\big)=\prod_{i\in\operatorname{supp}(r)}2^{n_{i}}\]
for any nonzero $r=(r_{1},\ldots, r_{m})$, thus the statement follows by Theorem \ref{thm:lowerbounds}.
\end{proof}

\section{Small semisimple groups of type $B$}\label{sec:small}

In this section, we consider the essential dimension of a semisimple group $G$ of type $B$ which is not covered by Theorem \ref{thm:main}. The quotients of small groups by the maximal central subgroups as in Proposition \ref{prop:generic} will be called \emph{small quotients}. We first calculate the essential dimension of the simplest cases of such $G$, namely, $G=\gSpin(3)^{m}/\gmu$ for any $m\geq 1$, where $\gmu$ is the diagonal central subgroup of $\gSpin(3)^{m}$ (see Proposition \ref{prop:sec51}). Secondly, we calculate the essential dimension of the smallest two remaining small quotients, namely, $G=(\gSpin(3)\times \gSpin(5))/\gmu$ or $(\gSpin(3)\times \gSpin(7))/\gmu$, where $\gmu$ is the diagonal central subgroup (see Proposition \ref{prop:smallesttwo}). Moreover, we provide a lower bound of the essential dimension of every small quotient introduced in the preceding section in Lemmas \ref{lem:smalllower1} and \ref{lem:smalllower2}.

We shall use the following result to obtain lower bounds for essential dimensions of groups mentioned above.

\begin{theorem}\cite[Theorem 7.8]{Reichstein00a}\label{thm:lowersmall}
	Let $G$ be a connected reductive group over $F$ and let $H$ be a finite abelian group such that the centralizer $C_{G}(H)$ is finite. Then, $\ed(G)\geq \rank(H)$.
\end{theorem}

In order to find a finite abelian group of a small quotient, we consider the extraspecial $2$-subgroup $\Delta(n)$ of $\gSpin(n)$  defined by the preimage of the group of diagonal matrices $D(n)$ of $\gSO(n)$ under the natural homomorphism $\gSpin(n)\to \gSO(n)$ so that we have an exact sequence
\begin{equation}\label{deltaexact}
1\to \gmu_{2}\to \Delta(n)\stackrel{\pi}{\to} D(n)\to 1.
\end{equation}

Let $c(1),\ldots, c(n)$ be the generators of the Clifford algebra of a split quadratic form of rank $n$ satisfying the relations
\begin{equation*}
c(i)^{2}=-1,\quad c(i)c(j)+c(j)c(i)=0,\quad 1\leq i\neq j\leq n.
\end{equation*}
For any subset $I=\{i_1, \ldots, i_k\}\subset [n]$ with an even number of elements $i_1<\cdots <i_k$, we set $c(I)=c(i_1,\ldots, i_k):=c(i_1)\cdots c(i_k)$ and $c(\emptyset)=1$. Then, we have
\begin{equation*}
	\Delta(n)=\{\pm c(I)\};
\end{equation*}
see \cite[\S 1]{Wood}. Moreover, the image of $c(I)$ under $\pi$ in (\ref{deltaexact}) is the diagonal matrix $d(I)$ with $-1$'s in positions $i_{1},\ldots, i_{k}$ and $1$'s in other diagonal positions.

In general, the essential dimension of semisimple groups of type $B$ can be bounded by the essential dimension of the corresponding even Clifford groups as in the following lemma, which can be viewed as a generalization of \cite[Lemma 3.2]{CM}.

\begin{lemma}\label{Lem:red}
	Let $G=\big(\prod_{i=1}^{m}\gSpin(2n_{i}+1)\big)/\gmu$ over $F$, where $\gmu$ is a central subgroup and let $G_{\red}=\big(\prod_{i=1}^{m}\Gamma^{+}(2n_{i}+1)\big)/\gmu$, where $\Gamma^{+}(2n_{i}+1)$ denotes the even Clifford group. Then, we have
	\[\ed(G_{\red})\leq \ed(G)\leq \ed(G_{\red})+m-\rank(\gmu).\]
\end{lemma}
\begin{proof}
	Consider the following commutative diagram with exact sequences
	\begin{equation*}
		\xymatrix{
			1 \ar@{->}[r] & (\gmu_{2})^{m}/\gmu \ar@{->}[r]\ar@{->}[d] & G \ar@{->}[r]\ar@{->}[d] & 
			\prod_{i=1}^{m}\gSO(2n_{i}+1)  \ar@{->}[r]\ar@{=}[d] & 1\\
			1 \ar@{->}[r] & (\gm)^{m}/\gmu \ar@{->}[r]  & G_{\red}\ar@{->}[r] & \prod_{i=1}^{m}\gSO(2n_{i}+1) \ar@{->}[r]& 1. \\
		}
	\end{equation*}	
	
	Then, this diagram induces an exact sequence
	\begin{equation*}
		H^{1}(K,(\gmu_{2})^{m}/\gmu)\to H^{1}(K,G)\to H^{1}(K, \prod_{i=1}^{m}\gSO(2n_{i}+1))\stackrel{\partial}{\to} H^{2}(K, (\gmu_{2})^{m}/\gmu) 
	\end{equation*}
	for any field extension $K/F$.
	As $\Ker(\partial)=H^{1}(K,G_{\red})$ (see \cite[Lemma 6.1]{Baek17}), we have 
	\begin{equation*}
		H^{1}(K,(\gmu_{2})^{m}/\gmu)\to H^{1}(K,G)\stackrel{\pi}{\twoheadrightarrow} H^{1}(K, G_{\red}).
	\end{equation*}
	By Lemma \ref{Lem:transitive}, the group $H^{1}(K,(\gmu_{2})^{m}/\gmu)$ acts transitively on the fibers of $\pi$. Hence, as $\ed((\gmu_{2})^{m}/\gmu)=m-\rank(\gmu)$, by \cite[Proposition 1.13]{BerhuyFavi}, the statement follows.
\end{proof}

\begin{lemma}\label{Lem:transitive}\cite[\S 28]{KMRT}
	Let $G$ be an algebraic group over $F$ and let $\gmu$ be a central subgroup of $G$. Then, the group $H^{1}(G,\gmu)$ acts transitively on the fibers of the map $H^{1}(K,G)\to H^{1}(K, G/\gmu)$ for any field extension $K/F$.
\end{lemma}

Now we present the first main result of this section.

\begin{proposition}\label{prop:sec51}
	Let $G=\gSpin(3)^{m}/\gmu$ over $F$, where $m\geq 2$ and $\gmu$ is the diagonal central subgroup of $\gSpin(3)^{m}$. Then, we have $\ed(G)=m+1$.
\end{proposition}
\begin{proof}
Let $G_{\red}=\Gamma^{+}(3)^{m}/\gmu$, where $\Gamma^{+}(3)$ denotes the even Clifford group. Then, by \cite[Lemma 6.1]{Baek17}, for any field extension $K/F$ we have
	\[H^{1}(K,G_{\red})=\{(q,\ldots, q)\,\,|\,\, q=\langle a, b, ab\rangle \text{ for some } a, b\in K^{\times}\}.\]
	Hence, $\ed(G_{\red})=2$, thus, by Lemma \ref{Lem:red} we obtain $\ed(G)\leq m+1$. The lower bound follows by Lemma \ref{lem:smalllower0} below with $n=1$.
\end{proof}

\begin{lemma}\label{lem:smalllower0}
    Let $G=\gSpin(2n+1)^m/\gmu$ over $F$, where $n\geq 1$, $m\geq 2$ and $\gmu$ is the diagonal central subgroup of $\gSpin(2n+1)^m$. Then, we have $\ed(G)\geq m+2n-1$.
\end{lemma}
\begin{proof}
Consider the following subgroup 
\begin{equation*}
    H'=\{(\pm c(I),\ldots, \pm c(I)) \,|\, I\subset [2n+1],\, |I| \text{ is even}\}
\end{equation*}
of $\Delta(2n+1)^m$, where $\Delta (2n+1)$ denotes the extraspecial $2$-subgroup of $\gSpin(2n+1)$. Then, $|H'|=2^{m+2n}$. Let $H=H'/\gmu$ and let $h_l=(1,\ldots,1,-1,1,\ldots,1)$, where $-1$ is placed in the $l$-th position and $1$'s are placed in other positions. Then, we have
\begin{equation*}
    H=\bigoplus_{i=1}^{2n} (\Z/2\Z)\overline{\big(c(i,i+1),\ldots,c(i,i+1)\big)}\bigoplus\big(\bigoplus_{l=1}^{m-1}(\Z/2\Z)\bar{h_l}\big).
\end{equation*}
Hence, $H$ is abelian with $\rank(H)=m+2n-1$. Since $c(I)$ corresponds to the diagonal matrix $d(I)$ in $\gSO(2n+1)$ and the intersection
\[\bigcap_{i=1}^{n-1}C_{\gSO(2n+1)^{m}}\big((d(i,i+1),\ldots, d(i,i+1))\big)\]
of the centralizers consists of $m$-tuples of the diagonal matrices in $\gSO(2n+1)$, the centralizer $C_G(H)$ is finite. Hence, by Theorem \ref{thm:lowersmall}, we have $\ed(G)\geq m+2n-1$, which completes the proof.
\end{proof}

\begin{remark}
The lower bounds in Lemma \ref{lem:smalllower0} are sharper than the lower bounds in Proposition \ref{prop:lower} if $n=1$ or $n=2$ and $m\leq 3$. In particular, we have
\begin{equation}\label{spinfive}
\ed\big(\gSpin(5)^{2}/\gmu\big)\geq 5
\end{equation}
for the diagonal central subgroup $\gmu$ of $\gSpin(5)^{2}$.
\end{remark}

Now we provide lower bounds for essential dimensions of all remaining small quotients in the following two lemmas. Note that all of the lower bounds of small quotients in (\ref{spinfive}) and Lemmas \ref{lem:smalllower1} and \ref{lem:smalllower2} are sharper than the lower bounds given by Proposition \ref{prop:lower}.

\begin{lemma}\label{lem:smalllower1}
	Let $G=\tilde{G}/\gmu$, where $\tilde{G}$ is one of the small groups in $(\ref{small2})$ and $\gmu$ is its central diagonal subgroup. Then, a lower bound for the essential dimension of $G$ is given by $:$
	
\begin{center}
	
	\begin{tabular}{ c | c }
		\hline
		$\tilde{G}$ & $\ed(G)$\rule{0pt}{2.5ex}\\
		\hline
		
		$\gSpin(3)\times \gSpin(5)$ & $\geq 4$ \rule{0pt}{2.5ex}\\
		
		$\gSpin(3)\times \gSpin(7)$  & $\geq 4$ \rule{0pt}{2.5ex}\\
		
		$\gSpin(3)\times \gSpin(9)$& $\geq 5$\rule{0pt}{2.5ex}\\
		
		$\gSpin(3)\times \gSpin(11)$ & $\geq 7$\rule{0pt}{2.5ex}\\
		
		$\gSpin(5)\times \gSpin(7)$ & $\geq 5$ \rule{0pt}{2.5ex}\\
		
		\hline
	\end{tabular}
	
\end{center}
\end{lemma}
\begin{proof}
Let $\tilde{G}=\prod_{j=1}^{2}\gSpin(2n_{j}+1)$ be one of the small groups in (\ref{small2}) with $n_1\leq n_2$ and let $\bar{G}=\prod_{j=1}^{2}\gSO(2n_{j}+1)$ be the corresponding adjoint group. By (\ref{spinfive}), we may assume that $n_1<n_{2}$. We set $G=\tilde{G}/\gmu$, where $\gmu$ is the diagonal central subgroup. We provide below a way to choose elements 
\begin{equation}\label{generatorhi}
h_{1}:=\big(c(I_{11}), c(I_{12})\big),\ldots ,h_{l}:=\big(c(I_{l1}),c(I_{l2})\big)\in \Delta(2n_{1}+1)\times \Delta(2n_{2}+1)
\end{equation}
for some subsets $I_{ij}\subset [2n_{j}+1]$ of even elements whose images in $G$ generate a finite abelian group $H$ of rank $l$ such that $C_{G}(H)$ is finite. Then, it follows by Theorem \ref{thm:lowersmall} that $\ed(G)\geq l$. 

We first observe that for any $j$ and subset $I\subset [2n_{j}+1]$ of even elements, the centralizer $C_{\gSO(2n_{j}+1)}\big(d(I)\big)$ of the diagonal matrix $d(I)$ becomes a block diagonal matrix consisting of two blocks corresponding to the partition $\{I,[2n_{j}+1]-I\}$ of $[2n_{j}+1]$. Hence, in order to obtain a finite centralizer $C_{G}(H)$ we set 
\begin{equation}\label{Itwo}
I_{12}=\{1,3,\ldots,2\cdot2\lceil \frac{n_{2}}{2} \rceil -1\},\, I_{i2}=\{2i-3,\, 2i-2\}
\end{equation}
for $2\leq i\leq n_{2}+1$ so that the centralizer $\cap_{i=1}^{n_{2}+1}C_{\gSO(2n_{2}+1)}\big(d(I_{i2})\big)$ consist of diagonal matrices in $\gSO(2n_{2}+1)$. Similarly, we set
\begin{equation}\label{Ione}
I_{11}=\{1,3,\ldots,2\cdot2\lceil \frac{n_{1}}{2} \rceil -1\},\, I_{k1}=\{1,\, 2\},\, I_{i1}=\{2(i+n_1-n_2)-3,\, 2(i+n_1-n_2)-2\}
\end{equation}
for $2\leq k\leq n_{2}-n_{1}+2$ and $n_{2}-n_{1}+3\leq i\leq n_{2}+1$. Then, we see that the images of $h_{1},\ldots h_{l}$ given by (\ref{Itwo}) and (\ref{Ione}) with $l=n_{2}+1$ generate a finite abelian group $H$ in $G$ with finite centralizer $C_{G}(H)$. These generators give the lower bounds for $\tilde{G}=\gSpin(3)\times \gSpin(7)$ or $\gSpin(3)\times \gSpin(9)$ in the statement.

We can improve the above lower bound by $1$ in the following two cases: If $h_{1}^{2}=\big(c(I_{11})^{2}, c(I_{12})^{2}\big)\neq \big(1,-1\big)$ in $G$, i.e., $\tilde{G}=\gSpin(3)\times\gSpin(5)$ or $\gSpin(3)\times\gSpin(11)$, we can add one more element $h_{l+1}:=\big(1,-1\big)$ to the set of generators (\ref{generatorhi}). Finally, if both $n_1,n_2>1$, i.e., $\tilde{G}=\gSpin(5)\times\gSpin(7)$, then we add one more element $h_{l+1}:=\big(c(4,4),c(5,7)\big)$ to the set of generators (\ref{generatorhi}), which completes the proof.\end{proof}

\begin{lemma}\label{lem:smalllower2}
	 Let $G=\tilde{G}/\gmu$, where $\tilde{G}$ is one of the small groups in $(\ref{small3})$, $(\ref{small4})$ and $\gmu$ is its maximal central subgroup. Then, a lower bound for the essential dimension of $G$ is given by $:$
    
\begin{center}
	
	\begin{tabular}{ c | c }
		\hline
		$\tilde{G}$ & $\ed(G)$\rule{0pt}{2.5ex}\\
		\hline
		
		$\gSpin(3)^3$ & $\geq 3$ \rule{0pt}{2.5ex}\\
		
		$\gSpin(3)^2\times \gSpin(5)$  & $\geq 4$ \rule{0pt}{2.5ex}\\
		
		$\gSpin(3)^2\times \gSpin(7)$& $\geq 5$\rule{0pt}{2.5ex}\\
		
		$\gSpin(3)^4$ & $\geq 5$\rule{0pt}{2.5ex}\\
		
		\hline
	\end{tabular}
	
\end{center}    
\end{lemma}
\begin{proof}
The proof is similar to the proof of Lemma \ref{lem:smalllower1}. Let $\tilde{G}=\prod_{i=1}^m\gSpin(2n_i+1)$ be one of the small groups in (\ref{small3}), (\ref{small4}) and let $\bar{G}=\prod_{i=1}^m\gSO(2n_i+1)$ be the corresponding adjoint group. We set $G=\tilde{G}/\gmu$, where $\gmu$ is the maximal central subgroup. We provide below elements $h_{1},\ldots, h_{l}$ in $\prod_{i=1}^{m}\Delta (2n_{i}+1)$ whose images $\bar{h}_{1},\ldots,\bar{h}_{l}$ in $G$ generate a finite abelian group $H$ of rank $l$ such that $C_{G}(H)$ is finite. Then, by Theorem \ref{thm:lowersmall} we have $\ed(G)\geq l$. For $1\leq i\leq l$, we write $d_{i}$ for the image of $h_{i}$ under the map $\pi$ in (\ref{deltaexact}). 

\medskip

\noindent Case: $m=3$. Then, $\gmu=\{(1,1,1), (-1,-1,1), (1,-1,-1), (-1,1,-1)\}$. For $\tilde{G}=\gSpin(3)^3$, we set the generators $h_{1},\, h_{2},\, h_{3}$ equal to
\begin{equation*}
\big(c(1,3), c(1,3),c(1,3)\big), \big(c(1,2), c(1,2),1\big), \big(c(1,2), 1, c(1,2)\big),
\end{equation*}
respectively. For $\tilde{G}=\gSpin(3)^2\times\gSpin(5)$, we set $h_{1}, h_{2}, h_{3}, h_{4}$ equal to
\begin{equation*}
\big(c(1,2),c(1,2),c(2,4)\big), \big(c(1,3),1,c(1,2)\big), \big(1,c(1,3), c(3,4)\big), \big(c(1,3),c(1,3),1\big),
\end{equation*}
respectively. Finally, for $\tilde{G}=\gSpin(3)^2\times\gSpin(7)$, set $h_{1}, h_{2}, h_{3}, h_{4}, h_{5}$ equal to
\begin{align*}
&\big(c(1,2), c(1,3),c(1,2)\big), \big(c(1,2), c(1,2), c(1,3,5,7)\big),  \big(1,c(1,3),c(3,4)\big),\\
&\big(c(1,2),c(1,3),c(1,2,5,6)\big), \big(c(1,3),1,c(2,5)\big),
\end{align*}
respectively. In each case, we obtain $H=(\Z/4\Z)\bar{h}_{1}\bigoplus \big(\bigoplus_{i=2}^{l}(\Z/2\Z)\bar{h}_{i}\big)$. Since the intersection $\cap_{i=1}^{l}C_{\bar{G}}(d_{i})$ consists of diagonal matrices in $\bar{G}$, $C_{G}(H)$ is finite as well.

\medskip

\noindent Case: $m=4$. Then, $\tilde{G}=\gSpin(3)^{4}$ and $$\gmu=(\Z/2\Z)(-1,-1,1,1)\times (\Z/2\Z)(1,-1,-1,1)\times (\Z/2\Z)(1,1,-1,-1).$$
We set generators $h_{1}, h_{2}, h_{3}, h_{4}, h_{5}$ equal to
\begin{align*}
	&\big(c(1,2),c(1,2),1,1)\big), \big(c(1,2),1,c(1,2),1)\big),  \big(c(1,2),1,1,c(1,2))\big),\\
	&\big(-1,1,1,1\big), \big(c(1,3),c(1,3),c(1,3),c(1,3))\big),
\end{align*}
respectively. Then, $H=\bigoplus_{i=1}^{5}(\Z/2\Z)\bar{h}_{i}$ and the centralizer $\cap_{i=1}^{5}C_{\bar{G}}(d_{i})$ consists of diagonal matrices in $\bar{G}$, thus $C_{G}(H)$ is finite.\end{proof}

We show the lower bound in Lemma \ref{lem:smalllower1} is sharp in the following two cases, which is the second main result of this section.

\begin{proposition}\label{prop:smallesttwo}
	Let $G=(\gSpin(3)\times \gSpin(5))/\gmu$ or $(\gSpin(3)\times \gSpin(7))/\gmu$ over $F$, where $\gmu$ is the diagonal central subgroup. Then, $\ed(G)=4$.
\end{proposition}
\begin{proof}
By Lemma \ref{lem:smalllower1}, it suffices to show that $\ed(G)\leq 4$ for both groups $G$.

\medskip

\noindent \textbf{Case:} $G=(\gSpin(3)\times \gSpin(5))/\gmu$. Let $G_{\red}=(\gGL(2)\times \gGSp(4))/\gmu$, where $\gGSp(4)$ denotes the group of symplectic similitudes. Then, by \cite[Lemma 6.6]{Baek17} we have
\[H^{1}(K,G_{\red})=\{\big((Q_{1},\gamma), (Q_{2}\tens Q_{3},\sigma)\big)\,\,|\,\, Q_{1}+Q_{2}+Q_{3}=0 \text{ in } \Br(K)      \} \]
for some quaternion $K$-algebras $Q_{i}$, $1\leq i\leq 3$, where $\gamma$ denotes the canonical involution on $Q_{1}$ and $\sigma$ denotes a symplectic involution on $Q_{2}\tens Q_{3}$. As $Q_{1}+Q_{2}+Q_{3}=0$ in $\Br(K)$, by a theorem of Albert we obtain 
\[Q_{1}=(a,b), Q_{2}=(a,c), \text{ and } Q_{3}=(a,bc)\]
for some $a,b,c\in K^{\times}$. Hence, by \cite[Proposition]{ST} we have
\[(Q_{2}\tens Q_{3}, \sigma)\simeq (M_{2}(K)\tens Q_{1}, \operatorname{ad}_{q}\tens \gamma)\]
for some adjoint involution $\operatorname{ad}_{q}$ on the matrix algebra $M_{2}(K)$ with respect to a $2$-dimensional quadratic form $q$. Therefore, $\ed(G_{\red})\leq 3$. Hence, by Lemma \ref{Lem:red} together with the exceptional isomorphisms, $\gGamma^{+}(3)\simeq \gGL(2)$ and $\gGamma^{+}(5)\simeq \gGSp(4)$, we have $\ed(G)\leq 4$.

Indeed, as the natural morphism
\[H^{1}(K,G_{\red})\to H^{1}(K, \gGL(2)^{3}/\gmu' ),\]
where $\gmu'$ is a maximal subgroup of $(\gmu_{2})^{3}$ given by the kernel of the product map $(\gmu_{2})^{3}\to \gmu_{2}$, is surjective for any field extension $K/F$, $\ed(G_{\red})\geq \ed(\gGL(2)^{3}/\gmu')$. By \cite[Theorem 1.3 (c)]{CR}, $\ed(\gGL(2)^{3}/\gmu')=3$, thus we obtain $\ed(G_{\red})=3$.

\medskip

\noindent \textbf{Case:} $G=(\gSpin(3)\times \gSpin(7))/\gmu$. Let $\tilde{G}= \gSpin(7)\times \gSL(2) \big(\simeq \gSpin(7)\times \gSpin(3) \big)$. Consider the representation 
\[\rho:\tilde{G}\to \GL(V),\] 
where $V=V(7)\tens V(3)$. Here, we view $\gSpin(7)$ as the subgroup of the spin group $\gSpin(8)$ of the symmetric bilinear form 
\begin{equation*}
b(x,y)=x^{T}\begin{pmatrix} 0 & I_{4}\\ I_{4} & 0\end{pmatrix}y
\end{equation*} 
on $V(7)$ with respect to a basis $e_{1},\ldots, e_{8}$,  i.e.,
\begin{equation*}
\gSpin(7)=\{a\in \gSpin(8)\,|\, \chi(a)(e_{4}-e_{8})=e_{4}-e_{8}\},
\end{equation*}
where $\chi:\gSpin(8)\to \gSO(8)$ denotes the vector representation. Let 
\begin{equation*}
v:=E_{11}+E_{52}\in V,
\end{equation*}
where $V$ is identified with the vector space of $7\times 2$ matrices and $E_{ij}$ denotes the matrix with a one in position $(i,j)$ and zeros elsewhere. Then, by \cite[Proposition 26]{SK} the identity component of the stabilizer of the vector $v\in V$ is isomorphic to  
\begin{equation}\label{identitycompofG}
\gSL(3)\times\gSO(2) \,\big(\simeq  \gSL(3)\times\gm\big),
\end{equation}
which is the component of a generic stabilizer in $\tilde{G}$. Let $\tilde{S}$ be the stabilizer of the point $[v]\in \P(V)$ in $\tilde{G}$. As $\dim(V/G)=1$, the point $[v]$ belongs to the open orbit in $\P(V)$ and the identity component $\tilde{S}^{\circ}$ of $\tilde{S}$ coincides with $(\ref{identitycompofG})$.

Let $a=(e_{1}-e_{5})(e_{2}-e_{6})(e_{3}-e_{7})(e_{4}+e_{8})$. Then, $a\in \gSpin(7)$ with $a^{2}=-1$ and
\begin{equation*}
\chi(a)=\begin{pmatrix} 0 & J\\ J & 0\end{pmatrix}\in \gSO(7), \text{where } J:=\diag(1,1,1, -1).
\end{equation*}
Let $b=\begin{pmatrix} 0 & i\\ i & 0\end{pmatrix}\in \gSL(2)$, where $i$ denotes a primitive $4$th root of $1$ and let $t=(a, b)\in \tilde{G}$. Then, $\rho(t)\cdot v=iv$, thus the element $t$ generates $\gmu_{4}\in \tilde{S}$. Since $\tilde{G}$ preserves a quartic form \cite[Proposition 26]{SK}, the stabilizer $\tilde{S}$ is generated by the group in (\ref{identitycompofG}) and $t$, thus
\begin{equation*}
\tilde{S}=(\gSL(3)\times\gm)\rtimes \gmu_{4}.
\end{equation*}
Let $S$ denote the stabilizer of the point $[v]\in \P(V)$ in $G$. Then, we have
\begin{equation*}
S=\tilde{S}/\gmu=(\gSL(3)\times\gm)\rtimes \gmu_{2}.
\end{equation*}

For a field extension $K/F$, we denote by $\gSL(3)_{\gamma}$ and $(\gm)_{\gamma}$ the twists of $\gSL(3)$ and $\gm$ by the $1$-cocycle $\gamma\in Z^{1}(K,\gmu_{2})$, respectively. As $\gmu_{2}$ acts trivially on the set $H^{1}(K, \gSL(3)_{\gamma}\times(\gm)_{\gamma})$, it follows by \cite[\S 5.5 Corollary 2]{Serre} that
\begin{equation*}
H^{1}(K, S)=\coprod_{[\gamma]\in H^{1}(K,\gmu_{2})}   H^{1}(K,\gSL(3)_{\gamma})\times H^{1}(K,(\gm)_{\gamma}),\end{equation*}
thus by \cite[Exercise 17.11]{Garibaldi} we obtain a surjection
\begin{equation}\label{surjectionone}
H^{1}(K, \gSO(3)\times(\gm)_{\gamma} \times  \gmu_{2})\to H^{1}(K, S).
\end{equation}
Since the twist $(\gm)_{\gamma}$ is a norm one torus of rank $1$ for a quadratic extension, we have $\ed((\gm)_{\gamma})=1$. As $\ed(\gSO(3))=2$ and $\ed(\gmu_{2})=1$, we have
\begin{equation}\label{sogmmutwo}
\ed(\gSO(3)\times(\gm)_{\gamma}\times  \gmu_{2})=4.
\end{equation}
By \cite[Theorem 9.3]{Garibaldi} the natural morphism 
\[H^{1}(K, S)\to H^{1}(K, G)\]
is surjective, thus by (\ref{surjectionone}) and (\ref{sogmmutwo}) we have
$\ed(G)\leq \ed(S)\leq 4$.\end{proof}

\end{document}